\documentclass[11pt]{article}

\usepackage[english]{babel}
\usepackage{authblk}

\title{Homology character of the parabolic coset poset}

\author[1]{Theo Douvropoulos}
\author[2]{Matthieu Josuat-Vergès}

\affil[1]{Brandeis University}
\affil[2]{IRIF, CNRS, Université Paris-Cité}

\date{}

\usepackage{amsmath,amssymb,amsthm,mathtools}
\usepackage[linktocpage,unicode]{hyperref}

\usepackage{tikz}

\usepackage{a4wide}

\usepackage{caption}

\hypersetup{
    pdftoolbar=true,        
    pdfmenubar=true,        
    pdffitwindow=false,     
    pdfstartview={FitH},    
    pdftitle={},
    pdfauthor={},
    pdfsubject={},
    pdfkeywords={},
    pdfnewwindow=true,      
    colorlinks=true,       
    linkcolor=red,          
    citecolor=blue,        
    filecolor=magenta,      
    urlcolor=cyan,           
    unicode=true          
}

\newtheorem{theo}{Theorem}[section]

\newtheorem{lemm}[theo]{Lemma}
\newtheorem{prop}[theo]{Proposition}

\theoremstyle{definition}
\newtheorem{defi}[theo]{Definition}
\newtheorem{rema}[theo]{Remark}

\newtheorem{nota}[theo]{Notation}

\DeclareMathOperator{\Inv}{Inv}
\DeclareMathOperator{\des}{Des}
\DeclareMathOperator{\asc}{Asc}

\DeclareMathOperator{\Span}{Span}
\DeclareMathOperator{\Ind}{Ind}
\DeclareMathOperator{\Fix}{Fix}
\DeclareMathOperator{\rk}{rk}
\DeclareMathOperator{\Fr}{Fr}

\newcommand{\bchi}{\boldsymbol{\chi}}
\newcommand{\bphi}{\boldsymbol{\varphi}}

\newcommand{\bxi}{\boldsymbol{\xi}}
\newcommand{\bepsilon}{\boldsymbol{\epsilon}}
\newcommand{\sign}{\mathbf{\epsilon}}
\newcommand{\triv}{\mathbf{1}}

\newcommand{\PP}{\mathcal{P}}
\newcommand{\LL}{\mathcal{L}}

\newcommand{\CC}{\mathcal{C}}
\newcommand{\DD}{\mathcal{D}}
\newcommand{\CCP}{\mathcal{C}^{>0}}

\begin{document}

\maketitle

\begin{abstract}
    Motivated by the analogy with the Coxeter complex on one side, and parking functions on the other side, we study the poset of parabolic cosets in a finite Coxeter group.  We show that this poset is Cohen-Macaulay, and get an explicit formula for the character of its (unique) nonzero homology group in terms of the Möbius function of the intersection lattice.  This homology character becomes a positive element of the parabolic Burnside ring (in its natural basis) after tensoring with the sign character.  The coefficients of this character essentially encode the colored $h$-vector of the positive chamber complex (following Bastidas, Hohlweg, and Saliola, this complex is defined by taking Weyl chambers that lie on the positive side of a generic hyperplane).  Roughly speaking, tensoring by the sign character on one side corresponds to the transformation going from the $f$-vector to the $h$-vector on the other side.
\end{abstract}

\section{Introduction}

Let $G$ be a group and $\mathcal{H}$ be a collection of subgroups of $G$.  The {\it  coset poset} associated to $G$ and $\mathcal{H}$ is defined as the set of cosets $gH$ (with $g\in G, \; H\in \mathcal{H}$), ordered by inclusion.  It was introduced by Abels and Holz~\cite{abelsholz}, who relate its topological properties with algebraic properties of $\mathcal{H}$.  Brown~\cite{brown} considers the coset poset in the case of a finite group $G$ (and $\mathcal{H}$ the collection of all proper subgroups), in connection with the probabilistic zeta function of $G$: the evaluation at $-1$ of this zeta function is the inverse of  the Möbius invariant of the coset poset.

In this work, we consider the case of a finite Coxeter group $W$ of rank $n$ (and its geometric realization as a reflection group in a real vector space $V \simeq \mathbb{R}^n$). 
\begin{itemize}
    \item By considering the collection of standard parabolic subgroups, the associated coset poset is the {\it Coxeter complex}.  It also identifies with the face poset of the $W$-permutohedron.  This link with the permutohedron shows that the proper part of the poset has the topology of a $(n-1)$-dimensional sphere.  Moreover, the character of $W$ acting on its $(n-1)$st homology group is the sign character (since each reflection reverses the orientation of the permutohedron).  
    \item By considering noncrossing parabolic subgroups, the associated coset poset is the {\it poset of parking functions}.  This poset has been introduced by Edelman~\cite{edelman} in type A, further studied in~\cite{delcroixogerjosuatvergesrandazzo,douvropoulosjosuatverges2}.  In the latter, we proved that it is Cohen-Macaulay and gave the character of $W$ acting on its $(n-1)$st homology group as a parking character tensored with the sign character. 
\end{itemize}
Motivated by these two examples, we go one step further and investigate the coset poset $\PP$ associated to the collection of all parabolic subgroups.   We show that $\PP$ is Cohen-Macaulay (Theorem~\ref{theo1}), by following the same method as with the parking functions poset~\cite{douvropoulosjosuatverges2}.  Consequently, the interesting homology group to study is $\tilde H_{n-1}(\bar \PP)$, where $\bar\PP := \PP - \{W\}$ is the proper part of $\PP$ (see Section~\ref{sec:defpp} for details). Denote by $\bxi$ the character of $W$ acting on this homology group.  It is an element of the {\it parabolic Burnside ring}, which has a natural basis $(\bphi_X)_{X \in \Theta}$ that will be used throughout (see Orlik and Solomon~\cite{orliksolomon}, and the review in~\cite[Section~3.1]{douvropoulosjosuatverges1}). The combinatorial nature of $\bxi$ is unraveled as follows: 
\begin{itemize}
    \item In Theorem~\ref{theo2}, we give a simple explicit formula for $\bxi$ in terms of the Möbius function of the intersection lattice, and the basis $(\bphi_X)_{X \in \Theta}$.  It is a straightforward application of the Whitney homology technique devised by Sundaram~\cite{sundaram}.
    \item Let $\bepsilon$ denote the sign character of $W$.  In Proposition~\ref{theo3}, we show that $\bxi\otimes\bepsilon$ is positive in the $(\bphi_X)_{X \in \Theta}$ basis, and show that the coefficients counts certain facets in the {\it positive chamber complex} $\CCP$.
\end{itemize}

The complex $\CCP$ is defined by taking all Weyl chambers that entirely lie in the positive side $H^+$ of some generic hyperplane $H$, following Bastidas, Hohlweg, and Saliola~\cite{bastidashohlwegsaliola}. Some basic properties of this construction will be given in Section~\ref{sec:poschambercomp}.  It follows from~\cite{greenezaslavsky} that the number of chambers in $\CCP$ is $\prod_{i=1}^n e_i $ where $e_i$ are the exponents of $W$ (assuming it is irreducible).  By building on this, we will see that $\bxi$ is essentially an ``almost-colored'' $f$-vector of $\CCP$, while $\bxi\otimes\bepsilon$ is the corresponding ``almost-colored'' $h$-vector.  The precise meaning of this will be given in Proposition~\ref{theo3}.  In the case where $H^\perp$ meets the fundamental chamber $F$, this $h$-vector has a nice interpretation in terms of descents (see Theorem~\ref{theo4}), similar to the one obtained by Stembridge~\cite{stembridge} in his study of Coxeter cones.

The case of the symmetric group is detailed in Section~\ref{sec:sym}.  Parabolic cosets correspond to combinatorial objects known as {\it uniform block permutations}, see Sequence~\href{https://oeis.org/A023998}{A023998} in the OEIS~\cite{oeis} and references therein. To our knowledge, the poset structure on uniform block permutations hasn't been considered before and most of the interest on these objects come from a monoid structure, see~\cite{fitzgerald}.  We give generating functions for the characters $\bxi$ and for their dimensions, in terms of a Bessel function.

\section{The parabolic coset poset}
\label{sec:defpp}

We keep $W$, $n$, and $V\simeq\mathbb{R}^n$ as in the introduction.  Let $\LL$ be the {\it intersection lattice} of $W$.  We recall that its minimal element is $V$, its rank 1 elements are the {\it reflecting hyperplanes} (such a hyperplane is the fixed point subspace $\Fix(t)$ of a reflection $t\in W$), and any other element is an intersection of reflecting hyperplanes.  Its maximal element is the $0$-dimensional subspace $\{0\}$. It is a geometric lattice. Moreover, it is isomorphic to the lattice of parabolic subgroups of $W$ (where the order is inclusion) via the map
\begin{align} \label{eq:isoLL}
    \LL \owns X \mapsto W_X := \{ w \in W \;:\; X\subset \Fix(w) \}.
\end{align}

\begin{defi}
    The {\it parabolic coset poset} is:
    \[
        \PP
        :=
        \biguplus_{X\in \LL} W/W_X
        =
        \big\{ w W_X \; : \; w\in W,\; X \in \LL \big\}.
    \]
The partial order is inclusion.  The action of $W$ is left multiplication. 
\end{defi}

Clearly, $\PP$ has a unique maximal element, namely $W$ itself.  The minimal elements of $\PP$ are the singletons $\{w\}$ for $w$ in $W$ (cosets modulo the trivial subgroup). 

\begin{prop} \label{prop:interval}
    For each $w\in W$, the interval $[\{w\},W]$ in $\PP$ is isomorphic to $\LL$ via the map
    \[
        w W_X \mapsto X.
    \]
\end{prop}

\begin{proof}
    Via the action of $w^{-1}$, the interval $[\{w\},W]$ is isomorphic to $[\{e\},W]$ where $e$ is the neutral element of $W$.  This interval contains precisely the parabolic subgroups of $W$.  It is isomorphic to $\LL$ via the inverse of the map in~\eqref{eq:isoLL}.
\end{proof}

\begin{prop} \label{prop:ordfilter}
    For each $wW_X \in \PP$, the order ideal
    \[
        \{ 
            w' W_{X'} \in \PP \;:\; w'W_{X'} \subset w W_X
        \}
    \]
    is isomorphic to the parabolic coset poset of $W_X$.
\end{prop}

\begin{proof}
    Via the action of $W$, we can assume $w=e$. First note that $w'W_{X'} \subset W_X$ implies $w' \in W_X$ and $W_{X'} \subset W_X$.  So $w'W_{X'}$ is a parabolic coset of $W_X$.  The result easily follows (the isomorphism mentioned in the proposition is just the identity).
\end{proof}

Let $\bar\PP := \PP - \{W\}$.  It is called the {\it proper part} of $\PP$, and is the interesting poset from the topological point of view: indeed, the topology associated to $\PP$  is trivial since it has a maximum (see~\cite[Lecture~1]{wachs}).  We also introduce $\hat \PP := \PP \cup \{ \varnothing\}$.  This is again a poset ordered by inclusion, obtained from $\PP$ by adding the minimal element $\varnothing$.  Its topological relevance stems from Philip Hall's Theorem: the reduced Euler characteristic of the order complex of $\bar \PP$ equals $\mu(\hat\PP)$, the total Möbius invariant of $\hat\PP$ (see~\cite[Lecture~1]{wachs}).

\begin{prop}
    The poset $\hat\PP$ is a lattice.    
\end{prop}

\begin{proof}
    The intersection of two parabolic cosets is either empty, or another parabolic coset.  This shows that meets exist in $\PP$.  The existence of joins follows.
\end{proof}

\section{Topology of the parabolic coset poset}

We refer to Wachs~\cite{wachs} and Kozlov~\cite{kozlov} for poset topology, in particular for the notion of shellability and EL-shellability.  We refrain from writing a detailed introduction to this subject, because our topological result is a direct application of known methods and closely follows our previous work~\cite{douvropoulosjosuatverges2} (in the case of parking functions).

As a geometric lattice, $\LL$ is a shellable poset.  We review the construction of an EL-labeling.  Let $\LL_1$ be the set of rank 1 elements in $\LL$, {\it i.e.}, the reflecting hyperplanes of $W$ in $V$.  Let $\prec$ be an arbitrary total order on $\LL_1$.  To each cover relation $X \lessdot Y$ in $\LL$, we associate the label
\begin{align} \label{eq:labelling}
    \lambda(X,Y) := \min_{\prec} \big\{ H \in \LL_1 \;:\; X \cap H = Y \big\}.
\end{align}
Each saturated chain $X_1\lessdot \dots \lessdot X_k$ thus gives a {\it label sequence} $(\lambda(X_1,X_2), \dots ,\lambda( X_{k-1}, X_k) )$.  This turns out to be an EL-labeling, which means that for each relation $X\leq Y$: i) there exist a unique saturated chain $X \lessdot \dots \lessdot Y$ having a $\prec$-increasing label sequence, ii) this increasing label sequence is minimal for the lexicographic order, among saturated chains for $X$ to $Y$.  The existence of the EL-labeling imply that $\LL$ is shellable.  More precisely, a shelling is obtained by sorting maximal chains of $\LL$ via the lexicographic order of associated label sequences.   We refer to Wachs~\cite[Section~3.2.3]{wachs} for details and background about this.

\begin{lemm} \label{lemm:shell}
    Let $A \subset \LL_1$ with $A \neq \varnothing$. The subposet 
    \[
        \LL^{(A)} := \{V\} \cup \{ X \in \LL \; : \;  \exists H \in A, \; H\leq X  \} 
    \]
    of $\LL$ is shellable. 
\end{lemm}

\begin{proof}
    Let $\prec$ be a total order on $\LL_1$ such that $A$ is a prefix.  Let $\lambda$ be the associated EL-labelling of $\LL$ defined in \eqref{eq:labelling}.  Note that for each $H \in \LL_1$, we have $V\lessdot H$ and $\lambda(V,H) = H$.  

    It is straightforward to check that the restriction of $\lambda$ is again a EL-labelling of $\LL^{(A)}$.  Alternatively, the maximal chains of $\LL^{(A)}$ form a prefix in the lexicographic order on maximal chains of $\LL$, and it follows that this prefix is a shelling order for $\LL^{(A)}$.
\end{proof}

Let $P$ be a ranked poset and let $r$ denote its rank ({\it i.e.}, every maximal chain in $P$ has cardinality $r+1$).  We call $P$ {\it spherical} if its order complex is homotopy equivalent to a wedge of $r$-dimensional spheres.  A classical tool to prove this property is shellability.  In particular, if $P$ is bounded ({\it  i.e.}, it has a minimal element $\hat 0$ and a maximal element $\hat 1$) and has an EL-labeling, then $\bar P := P-\{\hat 0,\hat 1\}$ is spherical.  For example, $\bar \LL$ is spherical.  The EL-labeling of $\LL$ naturally extends to an EL-labeling of $\PP$, but it is not clear how to get an EL-labeling of $\hat \PP = \PP\cup \{\varnothing\}$.  So, we use an alternative inductive method and the main tool is the following lemma.

\begin{lemm} \label{lemm:inducposet}
    Let $P$ be a ranked poset of rank $r$.  Suppose that $P = P_1 \cup P_2$ where
    \begin{itemize}
        \item $P_1$ and $P_2$ are spherical posets of rank $r$,
        \item $P_1\cap P_2$ is a spherical poset of rank $r-1$,
        \item every chain in $P$ is included in $P_1$ or $P_2$ (this condition means that the order complex of $P$ is geometrically the union of those of $P_1$ and $P_2$).
    \end{itemize}
    Then $P$ is spherical as well.
\end{lemm}

\begin{proof}
    This is the same as~\cite[Lemma~5.1]{douvropoulosjosuatverges2}.  The idea is to use the Mayer-Vietoris long exact sequence to compute the homology of $P$ in terms of those of $P_1$, $P_2$, and $P_1\cap P_2$.  The third condition above is the assumption that permits to use this technique.   We obtain that this homology is a free abelian group in degree $r$ (via the first two conditions above, and the fact that an extension of two free abelian groups is a free abelian group), and $0$ in other degrees.  This suffices to characterize a wedge of $r$-dimensional spheres.  See~\cite{douvropoulosjosuatverges2} for details. 
\end{proof}

Let us introduce some notations needed for the proof of the next proposition.  The Coxeter length  of $w\in W$ is denoted by $\ell(x)$.  The (right) inversion set of $w\in W$ is:
\[
    \Inv(w) 
        :=
    \big\{
        \Fix(t) \;:\; t\in T \text{ such that } \ell(wt) < \ell(w)
    \big\}.
\]
It is convenient to use reflecting hyperplanes rather than reflections here, in order to apply Lemma~\ref{lemm:shell} with $A = \Inv(w)$.  Also, recall that the right weak order on $W$ is defined by $w_1\leq w_2$ iff $\Inv(w_1) \subset \Inv(w_2)$.  

\begin{prop} \label{prop:hatPspherical}
    The poset $\bar \PP$ is spherical (of rank $n-1$).
\end{prop}

\begin{proof}
    We follow the scheme of proof in~\cite[Section~5]{douvropoulosjosuatverges2}.  For any nonempty subset $A \subset W$, define:
    \[
        \bar\PP[A] := 
        \Big\{
            w W_X \;:\; X \in \LL\backslash\{0\},\; w \in A
        \Big\} \subset \bar\PP.
    \]
    This is the {\it order filter} of $\bar \PP$ generated by the elements $\{w\}$ with $w\in A$.  Now, let $w_1,w_2,\dots $ be a linear extension of the right weak order (an indexing of the elements of $W$, such that $w_i \leq w_j$ in the right weak order implies $i\leq j$). Let $A_i := \{ w_j \;:\; 1 \leq j \leq i\}$.  The idea is to show that $\bar\PP[A_i]$ is spherical by induction on $i$, using Lemma~\ref{lemm:inducposet}.  The maximal value of $i$ is $|W|$, and clearly we have $A_{|W|} = W$ and $\bar \PP[W] = \bar\PP$.  Also note that $\bar \PP[A_1]$ has a minimal element (namely $w_1$), so that this poset is topologically trivial.  

    The induction is as follows.  Assume that $\bar \PP[A_i]$ is spherical of rank $n-1$.  The poset $ \bar \PP[\{w_{i+1}\}]$ is too, because it is topologically trivial.  Each chain in $\bar\PP[A_{i+1}] = \bar\PP[A_{i}] \cup \bar\PP[\{w_{i+1}\}]$ is included in either $\bar\PP[A_{i}]$ or $\bar\PP[\{w_{i+1}\}]$ (because the minimal element of the chain is either in $\bar\PP[A_{i}]$ or $\bar\PP[\{w_{i+1}\}]$).  The last condition to check in order to apply Lemma~\ref{lemm:inducposet} is that $\bar\PP[A_i] \cap \bar\PP[\{w_{i+1}\}]$ is spherical of rank $n-2$.  Lemma~\ref{lemm:isolemma} below permits to identify this poset with the proper part of $\LL^{(\Inv(w_{i+1}))}$.  This poset is shellable by Lemma~\ref{lemm:shell} ($\Inv(w_{i+1})\neq\varnothing$ because $w_1=e$ is the unique element with empty inversion set), so that its proper part is spherical. 
    
    This completes the induction showing that $\PP[A_i]$ is spherical for all $i$.
\end{proof}

\begin{lemm} \label{lemm:isolemma}
    In the notations of the above proof, there is an isomorphism:
    \begin{align} \label{eq:isopp}
        \PP[A_i] \cap \PP[\{w_{i+1}\}] 
            \to
        \LL^{( \Inv(w_{i+1}) )},
    \end{align}
    explicitly given by the map $w_{i+1} W_X \mapsto X$.  
\end{lemm}

\begin{proof}
Since this map is the restriction of the isomorphism $\PP[\{w_{i+1}\}] \to \LL$ (as in Proposition~\ref{prop:interval}), it remains to check:
    \begin{align} \label{eq:equival}
        w_{i+1} W_X \in \PP[A_i]
            \Longleftrightarrow
        X \in \LL^{( \Inv(w_{i+1}) )}. 
    \end{align}
    To do this, let $u$ be the minimal length representative of $w_{i+1} W_X$.  We have $u\leq w_{i+1}$ in the right weak order, moreover $u < w_{i+1}$ is equivalent to $u\in A_i$ (since, by construction, $A_i$ contains all elements smaller than $w_{i+1}$ in the right weak order).  Let us check that both sides of~\eqref{eq:equival} are equivalent to $u < w_{i+1}$:
    \begin{itemize}
        \item The left-hand side is equivalent to the existence of $v \in A_i$ such that $v\in w_{i+1} W_X$, which is equivalent to $u<w_{i+1}$ by unicity of the minimal length representative.
        \item The right-hand side is equivalent to the existence of $t\in T$ such that $\Fix(t)\in \Inv(w_{i+1})$ and $X \subset \Fix(t)$.  This is equivalent to the existence of $t\in W_X \cap T$, with $\ell(w_{i+1}t) < \ell(w_{i+1})$, {\it i.e.}, $w_{i+1}$ not being the minimal length representative in $W/W_X$.
    \end{itemize}
\end{proof}

For any poset $P$, let $\hat P$ denote the smallest bounded poset containing $P$, (obtained by adding a maximal/minimal element if it doesn't already have one.  (This notation agrees with $\hat\PP$ introduced earlier.)   A ranked poset $P$ is {\it Cohen-Macaulay} if every open interval in $\hat P$ is spherical.  For example, a shellable poset is Cohen-Macaulay.  See again~\cite{kozlov,wachs} for details.

\begin{theo}\label{theo1}
    The poset $\PP$ is Cohen-Macaulay.
\end{theo}

\begin{proof}
    Let $(x,y)$ be an open interval in $\hat \PP$. 
    
    First, consider the case where $x,y\in \PP$.  By Proposition~\ref{prop:interval}, $(x,y)$ is isomorphic to an open interval in $\LL$.  This open interval is spherical, by shellability of $\LL$.

    Consider the case where $x$ is the minimal element ($x=\varnothing$, in the convention of the previous section). The open interval $(x,y)$ in $\hat \PP$ is thus the principal order ideal generated by $y$ in $\PP$.  By Proposition~\ref{prop:ordfilter}, an order ideal $\{ w'W_{X'} \in \PP \;:\; w'W_{X'} \subsetneq w W_X \}$ identifies to the poset $\overline{P}$ attached to the parabolic subgroup $W_X$.  By induction on the rank of $W$, we can assume that it is spherical.

    The last case to check is that $\overline{\PP}$ is spherical.  This is Proposition~\ref{prop:hatPspherical}.

    We thus have proved that every open interval in $\hat \PP$ is spherical, so $\PP$ is Cohen-Macaulay by definition. 
\end{proof}

\section{Homology character}
\label{sec:homo}

Let us introduce some characters that were studied by Orlik and Solomon~\cite{orliksolomon}.  We also refer to the review in~\cite[Section~3.1]{douvropoulosjosuatverges1}.  For $X\in\LL$, let 
\begin{equation} \label{eq:def_phi}
    \bphi_X := \Ind_{W_X}^W (\triv),
\end{equation}
{\it i.e.}, the trivial character of $W_X$ induced on $W$.  This is the character of the representation $\mathbb{C}^{W/W_X}$ (with the natural left action of $W$ on $W/W_X$).  For example, $\bphi_V$ is the character of the regular representation, and $\bphi_{\{0\}}$ is the trivial character.  The {\it parabolic Burnside ring} of $W$ is the subring of the character ring of $W$ linearly generated by $(\bphi_X)_{X\in\LL}$.  Let $\Theta$ be a set of representatives for the orbits of $W$ acting on $\LL$.  Then $(\bphi_X)_{X\in\Theta}$ is a linear basis of the parabolic Burnside ring.  

Recall from the introduction that the character of $W$ acting on $\tilde H_{n-1}(\bar P)$ is denoted by $\bxi$.  Since $\PP$ is Cohen-Macaulay, $\bxi$ can be computed via {\it Whitney homology}.  This technique was introduced by Sundaram~\cite{sundaram}, see also Wachs~\cite[Section~4.4]{wachs}.

\begin{theo}\label{theo2}
    Let $\mu$ denote the Möbius function of $\LL$, and $\mu(X) = \mu(X,\{0\})$ for $X\in \LL$. The character of $W$ acting on $\tilde H_{n-1}(\bar\PP)$ is:
    \begin{align} \label{eq:formulaxi1}
        \bxi = (-1)^n \sum_{X\in \LL} \mu(X) \cdot  \bphi_X.
    \end{align}
    Alternatively, we have
    \begin{align} \label{eq:formulaxi2}
        \bxi = (-1)^n \sum_{X\in \Theta} \mu(X) \cdot [W:N(W_X)] \cdot \bphi_X,
    \end{align}
    where $N(W_X)$ is the normalizer of $W_X$ in $W$.
\end{theo}

\begin{proof}
    This is a direct application of the methods of Sundaram~\cite{sundaram}. We use the formulation in~\cite[Theorem~4.4.1]{wachs}, with the only {\it caveat} that we need to consider the dual poset $\PP^*$ to match the exact formulation of this reference (the dual poset is defined via $x\leq y$ in $\PP$ iff $x\geq y$ in $\PP^*$ and it satisfies $\mu_\PP(x,y)=\mu_{\PP^*}(y,x)$).
    
    First note that the parabolic subgroups $W_X$ are orbit representatives for the action of $W$ on $\PP$.  For any such $W_X$, the open interval $(W_X,W)$ in $\PP$ identifies to the interval $(X,\{0\})$ in $\LL$ via the inverse of the map in~\eqref{eq:isoLL}.  Since $\LL$ is shellable (and consequently Cohen-Macaulay), the homology of the interval $(X,\{0\})$ is given by:
    \[
        \tilde H_i\big( (X,\{0\}) \big)
            =
        \begin{cases}
            \mathbb{Z}^{|\mu(X)|} & \text{ if } i=\dim(X)-2, \\
            0 & \text{ otherwise}.
        \end{cases}
    \]
    Note also that $W_X$ acts trivially on $(X,\{0\})$ (for $Y$ in this interval and $w\in W_X$, we have $ Y \subset X \subset \Fix(w)$).  By definition, $W_X$ thus contributes to the $k$th {\it Whitney homology group} $WH_k(\PP)$ (where $k=\dim(X)$) by a term $|\mu(X)|\cdot \bphi_X = (-1)^{\dim(X)} \mu(X) \cdot \bphi_X$, so that
    \begin{align} \label{eq:WH}
        WH_k(\PP) 
            =
        \sum_{X \in \LL, \; \dim(X)=k}  (-1)^k \mu(X) \cdot \bphi_X.
    \end{align}
    The theorem by Sundaram mentioned at the beginning of this proof, in presence of the Cohen-Macaulay property (Theorem~\ref{theo1}), says that we have the following identity in the representation ring of $W$:
    \begin{align} \label{eq:sumWH}
        \tilde H_{n-1}(\bar\PP) 
            =
        \sum_{k=0}^{n-1} (-1)^{n+k} WH_k(\PP).
    \end{align}
    Combining~\eqref{eq:WH} and~\eqref{eq:sumWH} gives~\eqref{eq:formulaxi1}.

    The alternative formula in~\eqref{eq:formulaxi2} gives the coefficients of $\bxi$ in the basis $(\bphi_X)_{X \in \Theta}$.  To get this, note  that $\bphi_X = \bphi_Y$ if $X=w(Y)$ for some $w\in W$.  It remains only to show that the orbit of $X$ for the action of $W$ has cardinality $[W:N(W_X)]$.  This comes from the orbit-stabilizer theorem, because the stabilizer of $X$ is $N(W_X)$ (we have $w(X)=X$ iff $w W_X w^{-1} = W_X$).
\end{proof}

\begin{rema}[Rank-selected subcomplexes]\label{rem:rank_sel_subposets}
Our proof of Theorem~\ref{theo2} extends verbatim to a more general case, which we sketch here and return to in Remark~\ref{rem:equiv_h}. Recall the notation $\hat \PP = \PP \cup \{ \varnothing\}$; the poset $\hat \PP$ has rank $n+1$, with $\rk_{\hat\PP}(\emptyset)=0$, $\rk_{\hat \PP}(W)=n+1$, and in general, $\rk_{\hat\PP}(wW_X)=1+\operatorname{codim}(X)$. For any subset $R\subseteq [n]:=\{1,2,\ldots,n\}$, we define the \emph{rank-selected subposet} 
\[
\PP_R:=\{x\in\hat\PP:\ \rk_{\hat\PP}(x)\in R\}.
\]

\noindent Rank-selected subposets of Cohen-Macaulay posets are themselves Cohen-Macaulay \cite[Thm.~6.4]{baclawski} so that in particular, $\PP_R$ is spherical or rank $|R|-1$ (note that $\bar \PP=\PP_{[n]}$).

\noindent Similarly, we define rank-selected subposets of the intersection lattice $\LL$:
\[
\LL_R:=\{X\in\LL:\ 1+\operatorname{codim}(X)\in R\},
\]
and the posets $\LL_R\cup\{0\}$ with $\{0\}$ as a maximum element. We write $\mu_R$ for the M{\"o}bius function of $\LL_R\cup\{0\}$ and $\mu_R(X):=\mu_R(X,\{0\})$.

\noindent Again, $W$ acts on $\PP_R\cup\{W\}$ with orbit representatives given in terms of parabolic subgroups; such intervals in $\PP_R\cup \{W\}$ correspond to intervals in $\LL_R\cup\{0\}$; and the M{\"o}bius function is alternating in sign by \cite[Corol.~4.3]{stanley_alt}. Exactly as in the proof of Theorem~\ref{theo2} we get
\begin{align}
    \tilde H_{|R|-1}(\PP_R)=(-1)^{|R|}\sum_{X\in\LL_R\cup\{0\}}\mu_R(X)\cdot\bphi_X.
\end{align} 
\end{rema}

We announced in the introduction that $\bxi \otimes \bepsilon$ is positive in the basis $(\bphi_X)_{X\in \Theta}$.  One of the first related questions is to find the multiplicity of the trivial character in $\bxi \otimes \bepsilon$ (it corresponds to the number of orbits when we think of the associated action of $W$ on a set).  Clearly, this is also the multiplicity of $\bepsilon$ in $\bxi$.

\begin{prop}
    The multiplicity of $\triv$ and $\bepsilon$ in $\bxi$ are respectively $0$ and $(-1)^n\mu(\LL)$.
\end{prop}

\begin{proof}
    The multiplicity of $\triv$ in $\bphi_X$ is $1$ (because it is the character of $W/W_X$, which contains a unique orbit).  The multiplicity of $\sign$ in $\Phi_X$ is $1$ if $X = \varnothing$, $0$ otherwise (this can be deduced from the multiplicity of $\triv$ via Solomon's formula below, see~\eqref{eq:solomon}).

    It follows then from~\eqref{eq:formulaxi1} that the multiplicity of $\triv$ in $\bxi$ is $(-1)^n \sum_{X\in\LL} \mu(X) = 0$, and the multiplicity of $\bepsilon$ is $(-1)^n \mu(V) = (-1)^n \mu(\LL)$. 
\end{proof}

\begin{rema}
    Assume that $W$ is irreducible, and let $e_1,\dots,e_n$ be its exponents.  Then $\mu(\LL) = (-1)^n \prod_{i=1}^n e_i$.  
\end{rema}

By the previous proposition, the multiplicity of $\triv$ in $\bxi\otimes\bepsilon$ is $|\mu(V)|$.  By writing
\[
    \bxi\otimes\bepsilon
        =
    \sum_{X\in \Theta} c_X \bphi_X,
\]
it follows that $\sum_{X\in\Theta} c_X = |\mu(V)| = \mu(\LL)$.  To prove that $c_X \geq 0$ and find a combinatorial interpretation, it is natural to look for combinatorial sets with cardinality $|\mu(\LL)|$.  A natural candidate is the set of {\it cuspidal elements} in $W$ ({\it i.e.}, those $w\in W$ such that $\dim(\Fix(w))=0$).  But what really works here is the simplicial complex defined in the next section.

\section{The positive chamber complex}
\label{sec:poschambercomp}

We mentioned above that the Coxeter complex $\CC$ can be seen as a coset poset (with respect to standard parabolic subgroups).  But we take here a geometric point of view in terms of chambers.  Let $F \subset V$ be the fundamental chamber, so that the other chambers are the images $w(F)$ for $w\in W$.  The closure $\overline{F}$ is a full-dimensional simplicial cone in $V$, and its faces are also simplicial cones.   Let $\CC$ denote the {\it Coxeter complex}.  By definition, every face $f\in \CC$ can be written $w(f')$ where $w\in W$ and $f'$ is a face of $\overline{F}$.  We refer to~\cite{bjornerbrenti}.  The geometric realization of $\CC$ as an abstract simplicial complex is a triangulation of the $(n-1)$-dimensional sphere, which is obtained explicitly by intersecting the simplicial cones in $\CC$ with the unit sphere of $V$.  This triangulation of the $(n-1)$-dimensional sphere will be used to work out concrete examples of $\CCP$ in rank 3.

\begin{rema} 
    Below, the dimension $\dim(f)$ of a face $f\in \CC$ means $\dim_{\mathbb{R}}(\Span(f))$.   We use the same convention  for the subcomplex $\CCP \subset \CC$.  If we were to consider $\CC$ as an abstract simplicial complex, the dimension of $f\in\CC$ as a simplex would be $\dim_{\mathbb{R}}(\Span(f))-1$.
\end{rema}
\begin{rema}
    The $1$-dimensional cones in $\CC$, ({\it i.e.}, its vertices as an abstract simplicial complex) are called {\it rays}.  In the crystallographic case, each ray is the positive span of a weight (in the sense of Lie theory), and rays of $\overline{F}$ correspond to fundamental weights.  
\end{rema}

Let $\rho \in F$ be a generic vector.  Here, {\it generic} means that $\rho$ isn't orthogonal to any ray of $\CC$.  Let $H := \rho^{\perp}$, and
\begin{align} \label{def:H}
    H^- := \big\{ \alpha \in \mathbb{R}^n \;:\; \langle \alpha | \rho \rangle \leq 0 \big\},  
    \qquad
    H^+ := \big\{ \alpha \in \mathbb{R}^n \;:\; \langle \alpha | \rho \rangle \geq 0 \big\}.
\end{align}

\begin{defi}
    The {\it positive chamber complex} is defined by
    \[
        \CCP :=
        \{
            f \in \CC \;:\; f \subset H^+
        \}.
    \]
    It is a subcomplex of the Coxeter complex. 
\end{defi}

Let us give examples in rank 3, where the complex is $2$-dimensional.  We can see $\CC$ as a triangulation of the $2$-dimensional unit sphere in a $3$-dimensional real vector space, then use stereographic projection on the plane. Each hyperplane thus becomes a circle or a straight line.  The examples are in Figures~\ref{fig:ccpA3_1},~\ref{fig:ccpA3_2},~\ref{fig:ccpB3_1},~\ref{fig:ccpB3_2}.
The generic hyperplane is the red circle.  The gray chamber is the fundamental chamber, and the rest of $\CCP$ is in light gray.  These examples in rank 3 lead to the following (negative) observations.
\begin{itemize}
    \item The combinatorial structure of $\CCP$ depends on the chosen generic vector $\rho$.  For example, the two examples in type $A_3$ (respectively, $B_3$) are non-isomorphic complexes. 
    \item The complex $\CCP$ might not be convex in $\CC$.  In all examples but the first one, there is a pair of chambers in $\CCP$ and a minimal gallery between them which doesn't stay in $\CCP$.
\end{itemize}

We can already see why $\rho$ is required not to be orthogonal to any ray.  Indeed, this condition ensures that each chamber $C$ of $\CC$ is such that: 
\begin{itemize}
    \item either $C$ is in the interior of $H^+$ or $H^-$, 
    \item or $H$ meets (the interior of) $C$. 
\end{itemize}
Consequently, a perturbation of $\rho$ (equivalently, of $H$) does not change the fact that a given chamber is in $H^+$ or not.  This means that $\CCP$ is invariant under perturbation of $\rho$.  It is natural to introduce the hyperplane arrangement
\[
    \DD := \Big\{ f^\perp \;:\; f \text{ ray of } \CC \Big\}
\]
in $V$, called the {\it dual arrangement} of $W$.  See Figure~\ref{fig:dualA3} in type $A_3$, and Figure~\ref{fig:dualB3} in type $B_3$.  The regions of the dual arrangement are open sets where $\CCP$ is constant (as a function of $\rho$).

Let us examine what happens when $\rho$ crosses an hyperplane of $\mathcal{D}$.  Denote by $\rho_1$ the choice taken in Figure~\ref{fig:ccpA3_1}, $\rho_2$ the choice taken in Figure~\ref{fig:ccpA3_2}, and $\rho'$ the point where a shortest path from $\rho_1$ to $\rho_2$ crosses an hyperplane of $\mathcal{D}$.  Their locations in $F$ is approximately as follows (repeating the picture in Figure~\ref{fig:dualA3}):
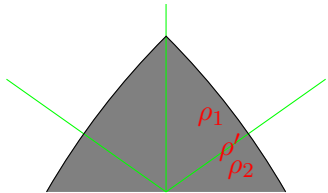
\[
    \begin{tikzpicture}[scale=5]
        \draw[fill,color=gray] (0.318,0) arc(30:45:2) arc(135:150:2);
        \draw[color=green] (0,0) -- (0,0.5);
        \draw[color=green] (0,0) -- (0.4243,0.3);
        \draw[color=green] (0,0) -- (-0.4243,0.3);
        \draw (0.318,0) arc(30:45:2) arc(135:150:2);
        \node[color=red] at(0.12,0.2){$\rho_1$};
        \node[color=red] at(0.2,0.06){$\rho_2$};
        \node[color=red] at(0.17,0.12){$\rho'$};
    \end{tikzpicture}
\]
Let $f$ be the unique ray such that $\rho' \in f^\perp$, $f \in \CCP(\rho_1)$, and $-f \in \CCP(\rho_2)$.  We have the following observations.
\begin{itemize}
    \item The facets in $\CCP(\rho_1)$ incident to $f$ are exactly those not in $\CCP(\rho_2)$.
    \item Similarly, the facets in $\CCP(\rho_2)$ incident to $-f$ are exactly those not in $\CCP(\rho_1)$
    \item There is a bijection between the facets described in the previous two items (explicitly, compose the antipodal map $\alpha\mapsto -\alpha$ defined on all chambers with the antipodal maps defined on chambers adjacent to $f$). 
    \item We have $\CCP(\rho') = \CCP(\rho_1) \cup \CCP(\rho_2)$, in particular $\CCP(\rho')$ contains strictly more facets than $\CCP(\rho_1)$ or $\CCP(\rho_2)$.
\end{itemize}
As a consequence, the number of facets of $\CCP(\rho)$ is minimal if and only if $\rho$ is generic. 

After these observations, we will (following~\cite{bastidashohlwegsaliola}) use the result below to obtain the number of facets of $\CCP$ in the generic case.

\begin{prop}[Greene and Zaslawky,~\cite{greenezaslavsky}] \label{prop:greenezaslawsky}
    Let $\mathcal{A}$ be an hyperplane arrangement in a real vector space, and $L$ the associated intersection lattice.  Let $H$ be a generic hyperplane in the same vector space, and $H^+$ one of the half-spaces defined by $H$.  Then, the Möbius invariant $\mu(L)$ is, in absolute value, the number of regions of $\mathcal{A}$ in $H^+$.
\end{prop}

So, the number of facets in $\CCP$ is $(-1)^n \mu(\LL)$ (in particular, it doesn't depend on $\rho$ as soon as $\rho$ is generic).  These facets are thus natural candidates to look for a combinatorial interpretation of the coefficients $c_X$ defined at the end of Section~\ref{sec:homo}.  The previous proposition will also be used in Section~\ref{sec:vectors} to take into account all faces of $\CCP$, not just facets.



\begin{prop} \label{prop:orderideal}
    If $\rho\in F$, the set
    \begin{align} \label{eq:setWCCP}
        \big\{ w \in W \;:\; w(\overline{F}) \in \CCP \big\}
    \end{align}
    is an order ideal in $W$ for the right weak order.
\end{prop}

\begin{proof}
    By definition, a cover relation in the right weak order is $ws \lessdot w$ for $w\in W$ and $s\in S$ such that
    $\ell(ws) < \ell(w)$ (where $\ell$ is the Coxeter length).  Assuming $w(\overline{F}) \in \CCP$, we have to show $ws(\overline{F}) \in \CCP$.

    Let $\alpha_1,\dots,\alpha_n$ be vectors that span the $n$ rays of $w(\overline{F})$, indexed so that the positive span of $\alpha_2,\dots,\alpha_n$ is the common panel of $w(\overline{F})$ and $ws(\overline{F})$.  The $n$ rays of $ws(\overline{F})$ are $\alpha_2,\dots,\alpha_n$ and $t(\alpha_1)$, where $t$ is the unique reflection in $W$ that fixes $\alpha_2,\dots,\alpha_n$, {\it i.e.}, $t = ws w^{-1}$.  Eventually, let $\beta \in V$ such that $t(\beta) = -\beta$ and $\langle \rho | \beta \rangle > 0$.
    
    The condition $ws \lessdot w$ means that $\overline{F}$ and $ws(\overline{F})$ are on the same side of the hyperplane $\Fix(t)$.  So, $\langle t(\alpha_1) | \beta \rangle$ has the same sign as $\langle \rho | \beta \rangle$, {\it i.e.},  $\langle t(\alpha_1) | \beta \rangle >0$.   We thus have $\langle \alpha_1 | \beta \rangle <0$.

    We have:
    \[
        t(\alpha_1)
            =
        \alpha_1 - 2 \frac{\langle\alpha_1|\beta\rangle}{\langle\beta|\beta\rangle} \beta.
    \]
    So, 
    \[
        \langle t(\alpha_1) |\rho \rangle
            =
        \langle\alpha_1|\rho\rangle - 2 \frac{\langle\alpha_1|\beta\rangle}{\langle\beta|\beta\rangle} \langle\beta|\rho\rangle.
    \]
    The first term is positive because $\alpha_1$ spans a ray of $w(\overline{F})$ which is in $\CCP$. We have seen that $\langle\alpha_1|\beta\rangle <0$ and $\langle\beta|\rho\rangle >0$.  We get $\langle t(\alpha_1) |\rho \rangle >0$ from the previous equation.  The $n$ rays of $ws(\overline{F})$ are thus in $\CCP$.  It follows that $ws(\overline{F}) \in \CCP$ as required.
\end{proof}

\begin{theo} \label{theo:shelling}
    The complex $\CCP$ is shellable. 
\end{theo}

\begin{proof}
    Via the action of $W$, we can assume $\rho\in F$ without changing the combinatorial structure of $\CCP$. So we can use Proposition~\ref{prop:orderideal}.  

    Let $w_1,w_2,\dots$ be a linear extension of the right weak order on $W$. By~\cite[Theorem~2.1]{bjorner}, the sequence of facets $w_1(\overline{F}), w_2(\overline{F}), \dots$ is a shelling of $\CC$.
    
    The linear extension can be chosen so that the elements $w$ satisfying $w(\overline{F}) \in \CCP$ form a prefix, since these elements form an order ideal (by Proposition~\ref{prop:orderideal}).  By definition, a prefix of a shelling order is a shelling order of the corresponding subcomplex.  So $w_1(\overline{F}), w_2(\overline{F}), \dots$ (where we keep only the facets of $\CCP$) is a shelling of $\CCP$.
\end{proof}

\begin{theo} \label{theo:ball}
    The geometric realization of $\CCP$ is homeomorphic to a $(n-1)$-dimensional ball. 
\end{theo}

\begin{proof}
    As a consequence of shellability, any pair of chambers in $\CCP$ can be joined by a gallery that stays in $\CCP$.  Moreover, each $(n-2)$-dimensional face of $\CCP$ belongs to either 1 or 2 chambers of $\CCP$ (because this holds in $\CC$).  These conditions imply that $\CCP$ is a pseudomanifold with boundary.

    By Danaraj and Klee~\cite{danarajklee}, a shellable pseudo-manifold with boundary is PL-homeomorphic to a closed ball.  The result follows.  (Alternatively, we could see the stereographic projection of the geometric realization $\CCP$ as a star-convex with respect to a point in the fundamental chamber.)
\end{proof}

\section{\texorpdfstring{The $f$- and $h$-vectors}{The f- and h-vectors}}
\label{sec:vectors}

The $f$-vector of an $(n-1)$-dimensional simplicial complex is $(f_i)_{-1\leq i \leq n-1}$, where $f_i$ is the number of $i$-dimensional faces.  Its $h$-vector $(h_i)_{0\leq i \leq n}$ is then defined via the relation:
\begin{align} \label{eq:rel_fh}
    \sum_{i=0}^n f_{i-1} z^i
        =
    \sum_{i=0}^n h_i z^i(1+z)^{n-i}.
\end{align}
This $h$-vector is known to encode important information about the simplicial complex. A classical example states that the $h$-vector of the Coxeter complex $\CC$ are the coefficients of the $W$-Eulerian polynomial $\sum_{w\in W} z^{\operatorname{des}(w)}$.

Let $S$ be the set of simple reflections of $W$.  Define a ``coloring'' map from the vertices of $\CC$ to $S$ by $w W_{(s)} \mapsto s$ (where $W_{(s)}$ is the subgroup generated by $S\backslash\{s\}$). This shows that $\CC$ is a {\it balanced} complex, which means that the $n$ vertices of any facet are colored by the $n$ elements of $S$.  See~\cite{bjorner}.  The $h$-vector of a balanced complex can be refined via the notion of {\it flag $h$-vector}.  In the present context, the positive chamber complex $\CCP$ is again a balanced complex.  But the natural refinement of the $f$- and $h$- vectors is obtained using $\Theta$ rather than subsets of $S$. Indeed, each subset $I\subset S$ naturally gives an element of $X\in \Theta$ by the condition that $W_X$ is conjugate to the standard parabolic subgroup $W_I$ (generated by $I$).  The corresponding notion of $f$-vector and $h$-vector are what we called ``almost-colored'' above, and they can be seen as the coefficients of some element in the parabolic Burnside ring.  

The next theorem says that $\bxi$ is the (suitably signed) almost-colored $f$-vector of $\CCP$.  

\begin{theo} \label{theo:xi_fvec}
    We have:
    \begin{align} \label{eq:xi_as_fvec}
        \bxi = (-1)^n \sum_{f \in \CCP} (-1)^{\dim(f)} \cdot \bphi_{\Span(f)}.
    \end{align}
\end{theo}

\begin{proof}
    Let $\mathcal{A}$ denote the Coxeter arrangement $\{ \Fix(t) \;:\; t\in T\}$ in $V$. Let $X \in \LL$.  We count the number of $f \in \CCP$ with $\Span(f)=X$. To do this, consider the {\it restricted arrangement}:
    \[
        \mathcal{A}^X := \big\{
            X \cap Y \;:\; Y\in \mathcal{A}, \; \dim(X \cap Y) = \dim(X)-1
        \big\}.
    \]
    This is a hyperplane arrangement in $X$.  Its intersection lattice is the interval $[X,\{0\}]$ in $\LL$, so its total Möbius invariant is $\mu(X)$ (where $\mu$ is the Möbius function of $\LL$ as before). Moreover, the regions of $\mathcal{A}^X$ are the faces $f\in \CC$ with $\Span(f)=X$.  By Proposition~\ref{prop:greenezaslawsky} (using $H^+\cap X$ as a generic half-space in $X$), we obtain:
    \[
        \#\big\{ f \in \CCP \;:\; \Span(f)=X
        \big\}
            =
        |\mu(X)|
            =
        (-1)^{\dim(X)} \mu(X).
    \]
    It follows that the right-hand side of~\eqref{eq:xi_as_fvec} is equal to the right-hand side of~\eqref{eq:formulaxi1}.
\end{proof}

The next step is to identify $\bxi \otimes \bepsilon$ with the corresponding almost-colored $h$-vector.  Let us first introduce some tools for dealing with tensoring by the sign character.

If $I\subset S$, let $\bphi_I := \Ind_{W_I}^W (\triv)$.  This notation is compatible with the definition of $\bphi_X$ in~\eqref{eq:def_phi}, because the lattice of subsets of $I$ naturally embeds in $\LL$ via
\begin{equation} \label{eq:mapI}
    I \mapsto \bigcap_{s\in I} \Fix(s).
\end{equation}
(In terms of parabolic subgroups, this is the inclusion of standard parabolic subgroups in all parabolic subgroups.)
The notation in terms of subsets of $S$ is convenient to state {\it Solomon's formula}~\cite{solomon}:
\begin{align} \label{eq:solomon}
    \bphi_I \otimes \bepsilon
        =
    \sum_{J\subset I} (-1)^{\#J} \bphi_J.
\end{align}
However, in the present context it is better to forget about subsets of $S$ and have a formulation in terms of $\CC$ and $\LL$.

\begin{lemm} \label{lemm:tensoringsign}
    Let $f,g\in\CC$, such that $f\subset g$ and $\dim(g)=n$ ($g$ is a chamber).  We have:
    \begin{align} \label{eq:tensoringsign}
        \bphi_{\Span(f)} \otimes \bepsilon
            =
        \sum_{\substack{f'\in \CC, \\ f\subset  f' \subset g}} (-1)^{n - \dim(f')} \bphi_{\Span(f')}.
    \end{align}
\end{lemm}

\begin{proof}
    Using the action of $W$, we can assume that $g$ is the fundamental chamber.  The subsets of $S$ identify with the faces of $g$ through the map 
    \[
        I \mapsto g \cap \bigg( \bigcap_{s\in I} \Fix(s) \bigg),
    \]
    and this map reverses the order relations. Equation~\eqref{eq:tensoringsign} becomes~\eqref{eq:solomon} through this identification.
\end{proof}

\begin{prop}\label{theo3}
    Let $f_1,f_2,\dots$ be a shelling of $\CCP$.
    The {\it type} $\tau(f_i)$ of a facet of $f_i \in \CCP$ is defined as follows: let $p_1,\dots,p_k$ be the panels of $f_i$ that are not panel of $f_j$ for some $j<i$, then
    \[
        \tau(f_i) 
            :=
        \Span\Big( \bigcap_{j=1}^k p_j \Big)
            \in \LL.
    \]
    We have: 
    \begin{equation}  \label{eq:xiotimeseps}
        \bxi \otimes \bepsilon  
            =
        \sum_{ f \in \CCP, \; \dim(f)=n }  \bphi_{\tau(f)}.
    \end{equation}
\end{prop}

\begin{proof}
    Since $\bepsilon \otimes \bepsilon = \triv$, Equation~\eqref{eq:xiotimeseps} is equivalent to:
    \begin{equation} \label{eq:xiotimes}
        \bxi  
            =
        \sum_{ f \in \CCP,\; \dim(f)=n }  \bphi_{\tau(f)} \otimes \bepsilon.
    \end{equation}
    Let $\DD^{(j)}$ be the subcomplex of $\CCP$ having $f_1,\dots,f_j$ as facets.  Being a shelling means that the intersection of $\DD^{(j-1)}$ with $f_j$ is purely $(n-1)$-dimensional (it is a nonempty union of panels of $f_j$).  We will show that for $j\geq 1$, we have:
    \begin{equation} \label{eq:xiotimesrec}
        (-1)^n \sum_{f \in \DD^{(j)},\; f \notin \DD^{(j-1)}} (-1)^{\dim(f)} \cdot \bphi_{\Span(f)}
            =
        \bphi_{\tau(f_j)} \otimes \bepsilon
    \end{equation}
    (if $j=1$, the condition $f \notin \DD^{(j-1)}$ is omitted).  By summing over $j$, we get
    \begin{equation}
        (-1)^n \sum_{f \in \CCP} (-1)^{\dim(f)} \cdot \bphi_{\Span(f)}
        =
        \sum_{ f \in \CCP,\; \dim(f)=n }  \bphi_{\tau(f)} \otimes \bepsilon.
    \end{equation}
    This is equivalent to~\eqref{eq:xiotimes}, by Theorem~\ref{theo:xi_fvec}.  It remains only to prove~\eqref{eq:xiotimesrec}.
    
    For a fixed $j$, let $p_1,\dots,p_k$ be the panels of $f_j$ that aren't a panel of $f_i$ for $i<j$ ({\it i.e.}, $p_1,\dots,p_k \notin \mathcal{D}^{(j-1)}$), and let $p_{k+1},\dots,p_n$ be the other panels of $f_j$.  From the definition of a shelling, a face $f \in f_j$ is in $\mathcal{D}^{(j-1)}$ iff it is included in some panel among $p_{k+1},\dots,p_n$.  Since the faces of $f_j$ form a boolean lattice, it follows that the faces $f \in \DD^{(j)}$ such that $f \notin \DD^{(j-1)}$ are the faces of $f_j$ that can be written as an intersection $\cap_{i\in I} p_i$ where $I \subset \{1,\dots,k\}$.  These faces form the interval between $\cap_{i=1}^k p_i$ and $f_j$ in the face lattice of $f_j$, so that
    \begin{equation} \label{eq:xiotimesrec2}
        \sum_{f \in \DD^{(j)},\; f \notin \DD^{(j-1)}} (-1)^{\dim(f)} \cdot \bphi_{\Span(f)}
            =
        \sum_{ \cap_{i=1}^k p_i \subset f \subset f_j } (-1)^{\dim(f)} \cdot \bphi_{\Span(f)}.
    \end{equation}
    By Lemma~\ref{lemm:tensoringsign}, the latter expression is $ (-1)^n \bphi_{\tau(f_j)} \otimes \bepsilon$.  This proves~\eqref{eq:xiotimesrec}.
\end{proof}

We eventually arrive at the result announced in the introduction.  It is much more satisfying than the previous proposition, because it gives a formula for $\bxi\otimes\bepsilon$ that doesn't involve a shelling order.

The {\it ascents} and {\it descents} of $w \in W$ are two complementary subsets of $S$ defined, using the Coxeter length $\ell$, by:
\[
    \des(w) :=
    \{
        s \in S \;:\; \ell(ws) < \ell(w)    
    \},
    \qquad    
    \asc(w) :=
    \{
        s \in S \;:\; \ell(ws) > \ell(w)    
    \}.
\]
Recall that for $I\subset S$, we defined $\bphi_I := \Ind_{W_I}^W (\triv)$ (see the paragraph before Lemma~\ref{lemm:tensoringsign}).

\begin{theo} \label{theo4}
    If $\rho \in F$ (and $H$, $H^+$, $H^-$ are defined accordingly via~\eqref{def:H}), we have:
    \begin{align} \label{eq:theo_desc_asc}
        \bxi \otimes \bepsilon
        =
        \sum_{\substack{w\in W \\ w(F) \subset H^+}}
        \bphi_{\asc(w)}
        =
        \sum_{\substack{w\in W \\ w(F) \subset H^-}}
        \bphi_{\des(w)}.
    \end{align}
\end{theo}

\begin{proof}
    We only prove the first equality. The other is deduced by taking $-\rho$ and $-F$ in place of $\rho$ and $F$.  (In the case where the long element $w_\circ \in W$ is central and thus equal to $-I$, we have $w_\circ(H)=H$.  The two sums are then equal, because we can pair elements on each side via the map $w\mapsto w w_\circ$.)
    
    It remains to prove the first equality.  Let $f_1, f_2,\dots $ be the shelling of $\CCP$ given by Theorem~\ref{theo:shelling}.  The idea is to use Proposition~\ref{theo3} with this shelling.

    Let $f_j$ be a facet of $\CCP$.  It corresponds to $w\in W$ such that $f_j = w(\overline{F}) \subset H^+$. Let $p$ be a panel of $f_j$.  There is $s \in S$ such that the two chambers adjacent to $p$ are $w(F)$ and $ws(F)$.  Note that
    \[
        \Span(p) = \Fix(wsw^{-1}).
    \]
    By definition, $s$ is a descent of $w$ iff $\ell(ws)<\ell(w)$, iff $ws(\overline{F})$ precedes $w(\overline{F})$ in the shelling order, iff $p$ is a panel of $f_i$ with $i<j$.  Let $p_1, \dots , p_k$ be the panels of $f_j$ that are not panel of $f_i$ for some $i<j$. We thus have
    \[
        \tau(f_j) 
            = 
        \Span\bigg( \bigcap_{i=1}^k p_k \bigg) 
            = 
        \bigcap_{i=1}^k \Span(p_k)
            =
        \bigcap_{s \in \asc(w)} \Fix(wsw^{-1})
            =
        w\bigg(\bigcap_{s \in \asc(w)} \Fix(s) \bigg).
    \]
    Recall that we identified subsets of $S$ to some elements in $\LL$ via~\eqref{eq:mapI}.  It follows $\bphi_{\tau(f_j)} = \bphi_{\asc(w)}$.  We thus obtain the result by reformulation of Proposition~\ref{theo3}.
\end{proof}

\begin{rema}[Equivariant $h$-vectors for $\bar\PP$]\label{rem:equiv_h}
For a Cohen-Macaulay poset like $\bar\PP$, the $h$-vector of its order complex $\Delta(\bar\PP)$ records information about the top homology dimensions of its type-selected subposets. If the poset comes with a group action (a $W$-action in our case), this can be generalized to a statement about representations \cite{stanley_grp}:
\[
h_i=\sum_{R\subset[n]\atop |R|=i}\dim\big(\tilde{H}_{|R|-1}(\PP_R)\big)\quad\quad\text{and}\quad\quad\mathbf{h_i}:=_W\bigoplus_{R\subset[n]\atop |R|=i}\tilde{H}_{|R|-1}(\PP_R),
\]
where the $\mathbf{h_i}$, defined as direct sums of $W$-modules $\tilde{H}_{|R|-1}(\PP_R)$, form what is known as the \emph{equivariant $h$-vector} of $\Delta(\bar\PP)$. The right hand sides of the two equations above are usually denoted $h_R$ and $\mathbf{h_R}$ and form the flag (and equivariant flag) $h$-vector of $\Delta(\bar\PP)$ (which is balanced since $\bar\PP$ is ranked). The entries $h_n=h_{[n]}$ and $\mathbf{h_n=h_{[n]}}$ are respectively the dimension and the homology character of the group action on the top homology of $\Delta(\bar\PP)$.

\noindent After Remark~\ref{rem:rank_sel_subposets}, we have
\begin{align}
    \mathbf{h_R}=(-1)^{|R|}\sum_{X\in\LL_R\cup\{0\}}\mu_R(X)\cdot\bphi_X.
\end{align}
The entries $h_i$ of the $h$-vector sum up to the number of maximal chains of the poset, which in our case equals $|W|\cdot \dfrac{|W|\cdot n!}{2^n}$ by \cite[Prop.~4.3]{chapuydouvropoulos}. However, no direct combinatorial interpretation is known for them. Moreover, a direct generalization of Theorem~\ref{theo4} does not hold. We calculate the equivariant $h$-vector for $A_3$ in Remark~\ref{rem:equiv_h2} and observe that it is not (alone or tensored by sign) a positive sum in the $\bphi$-basis.
\end{rema}

\section{The case of the symmetric group}
\label{sec:sym}

Let $\Pi_n$ denote the lattice of set partitions of $[n]$. The order is reverse refinement (the finest/coarsest set partition is the bottom/top element).  It is well-known that the intersection lattice for $\mathfrak{S}_n$ is isomorphic to $\Pi_n$.  The explicit construction is as follows.  The geometric representation is $V = \{ (v_1,\dots,v_n) \in \mathbb{R}^n \;:\; \sum v_i =0 \}$ where $\mathfrak{S}_n$ acts by permuting the $n$ coordinates:
\[
    \sigma \cdot (v_1,\dots,v_n) :=
    (v_{\sigma^{-1}(1)},\dots,v_{\sigma^{-1}(n)}).
\]
To $\pi\in \Pi_n$, we associate the subspace of $V$ defined by the equations $v_i=v_j$ for each pair $i,j$ of integers that are in the same block of $\pi$.  This is an isomorphism between $\Pi_n$ and the intersection lattice of $\mathfrak{S}_n$.

\begin{defi}
    A {\it uniform block permutation} (of size $n$) is a triple $(\pi_1,\pi_2,\lambda)$ where
    \begin{itemize}
        \item $\pi_1, \pi_2 \in \Pi_n$,
        \item $\lambda$ is a bijection from (the blocks of) $\pi_1$ to (those of) $\pi_2$,
        \item for each block $b\in \pi_1$, $\# b = \# \lambda(b)$.  
    \end{itemize}
    The order relation is given by $(\pi'_1,\pi'_2,\lambda') \geq (\pi_1,\pi_2,\lambda)$ if: i) $\pi'_1 \geq \pi_1$ and $\pi'_2 \geq \pi_2$ in $\Pi_n$, ii) $\lambda'$ is the natural quotient map obtained from $\lambda$.
\end{defi}

For example, a uniform block permutation of size $9$ is given by the pair of set partitions
\[
    (178|25|349|6, 237|14|569|8),
\]
where $\lambda$ is implicitly given by matching blocks with the same index: $178\mapsto 237$, $25 \mapsto 14$, {\it etc}. A smaller element is
\[
    (18|7|25|34|9|6, 
     37|2|14|69|5|8).
\]
Minimal elements naturally correspond to permutations of $\{1,\dots,n\}$ by 
\[
    \mathfrak{S}_n \owns \sigma 
        \mapsto
    (1|2|3|\dots |n, \sigma_1 | \sigma_2 | \sigma_3 | \dots |\sigma_n).
\]

The poset of uniform block permutations coincides with the type $A$ parabolic coset poset, and the correspondence is as follows.  The parabolic subgroups of  $\mathfrak{S}_n$ are indexed by $\Pi_n$ and are the products
\[
    \mathfrak{S}_\pi := \prod_{\beta \in \pi} \mathfrak{S}_{\beta}, \; (\pi\in \Pi_n),
\]
together with their natural action on $\{1,\dots,n\}$ ($\mathfrak{S}_{\beta}$ permutes the elements of $\beta$).  For each $\sigma \in \mathfrak{S}_n$ and $\pi \in \Pi_n$, the parabolic coset $\sigma \mathfrak{S}_\pi$ is associated to the uniform block permutation $(\pi,\sigma(\pi), \bar\sigma)$ where $\bar\sigma$ is the quotient map induced by $\sigma$.  

Recall that characters of $\mathfrak{S}_n$ are in correspondence with degree $n$ symmetric functions, via the Frobenius characteristic $\Fr$.  It is such 
that
\[
    \Fr( \bphi_\pi ) = H_\lambda, 
        \qquad
    \Fr( \bphi_\pi \otimes \bepsilon ) = E_\lambda,
\]
where $H_\lambda$ and $E_\lambda$ are respectively the homogeneous and elementary symmetric functions, and $\lambda$ is the (integer) partition of $n$ obtained by sorting block sizes of $\pi$.  We also use the involution $\omega$ defined by:
\[
    \omega( E_\lambda ) = H_\lambda, \qquad
    \omega( H_\lambda ) = E_\lambda,
\]
so that $\Fr(\bchi \otimes \bepsilon) = \omega( \Fr(\bchi) )$ for any character $\bchi$ of $\mathfrak{S}_n$.

It is convenient to write $H_\pi := H_\lambda$ and $E_\pi := E_\lambda$ (with $\pi$ and $\lambda$ as above).  More generally, we use the following:

\begin{nota}
    For any sequence $(X_i)_{i \geq 1}$ indexed by positive integers, and a set partition $\pi \in \Pi_n$, we denote
    \[
        X_\pi := \prod_{\beta \in \pi} X_{\#\beta}.
    \]
\end{nota}

\begin{lemm}[The moment-cumulant formula] \label{lemm:momentcumulant}
    Consider two formal power series
    \[
        M(x) = \sum_{n\geq 0} M_n \frac{x^n}{n!},
            \qquad
        K(x) = \sum_{n\geq 1} K_n \frac{x^n}{n!},
    \]
    with $M_0=1$.  The following conditions are equivalent:
    \begin{itemize}
        \item $K(x) = \log(M(x))$,
        \item $\forall n \geq 1, \; M_n = \sum_{\pi \in \Pi_n} K_\pi$,
        \item
        $\forall n \geq 1, \; K_n = \sum_{\pi \in \Pi_n} \mu(\pi) K_\pi$.
    \end{itemize}
\end{lemm}

(Moments and cumulants refer to quantities traditionally considered in probability theory.)

Let $\xi_n$ be the symmetric function $\Fr(\bxi)$ for the group $\mathfrak{S}_n$. 

\begin{prop}
    We have
    \begin{equation} \label{eq:gfS_n}
        \sum_{n\geq 1} \xi_n \frac{z^n}{n!}
            =
        -\log\bigg( \sum_{n\geq 0} H_n \frac{(-z)^n}{n!} \bigg).
    \end{equation}
\end{prop}

\begin{proof}
    By the general formula for $\bxi$ in Theorem~\ref{theo2}, we have:
    \[
        (-1)^{n-1} \xi_n = \sum_{\pi\in\Pi_n} \mu(\pi) H_\pi
    \]
    for all $n\geq 1$.   By Lemma~\ref{lemm:momentcumulant}, this is equivalent to the generating function in~\eqref{eq:gfS_n}.
\end{proof}

Note also that the Möbius function of $\Pi_n$ is given by
\begin{align} \label{eq:pin_mob}
    \mu(\pi):=\mu(\pi,[n])
        =
    (-1)^{\ell(\pi)} (\ell(\pi)-1)!.
\end{align}
An explicit formula for $\xi_n$ follows.

Let us now describe one possible choice of $\rho$ having particularly nice properties: $\CCP$ is convex in $\CC$, and its facets have a simple combinatorial description.  Let $(e_i)_{1\leq i \leq n}$ be the canonical basis of $\mathbb{R}^{n}$.  The simple roots for $\mathfrak{S}_n$ are the vectors $\alpha_i = e_i - e_{i+1}$ ($1\leq i \leq n-1$). Define
\[
    \omega_i := \frac{1}{n} \Big( \underbrace{n-i,\dots,n-i}_\text{$i$ terms},
    \underbrace{-i,\dots,-i}_\text{$n-i$ terms} \Big) \in V,
\]
for $1\leq i \leq n-1$.  These are the type $A$ fundamental weights.  In particular, we have
\[
    \overline{F}
        =
    \Big\{ (v_1 , \dots, v_n) \in V \;:\; v_1 \geq \dots \geq v_n \Big\}
        =
    \Span_{\mathbb{R}^+} \Big\{ \omega_i \;:\; 1\leq i \leq n\Big\}.
\]

\begin{prop} \label{prop:CCP_Sn}
    Let $\rho = \sum_{i=1}^{n-1} r_i \omega_i \in F$,
    where $r_i$ are positive coefficients, and $H$, $H^+$ be defined in terms of $\rho$ as in~\eqref{def:H}.  If $r_1$ is sufficiently large compared to $r_2,\dots,r_n$, we have $\sigma(F) \subset H^+$ if and only if $\sigma(1) = 1$ (for any $\sigma \in \mathfrak{S}_n$).
\end{prop}

\begin{proof}
    Since $F$ is the positive span of $(\omega_i)_{1\leq i \leq n}$, we have:
    \[
        \sigma(F) \subset H^+ 
            \;\Longleftrightarrow\;
        \forall i, \sigma(\omega_i) \subset H^+
            \;\Longleftrightarrow\;
        \forall i, \langle \sigma(\omega_i) \,|\, \rho \rangle \geq 0
            \;\Longleftrightarrow\;
        \forall i, \langle \omega_i \,|\, \sigma^{-1}(\rho) \rangle \geq 0.
    \]
    Explicit computations give:
    \[
        \langle \omega_1 \;|\; \sigma^{-1}(\omega_1) \rangle
            =
        \begin{cases}
             1 - \frac{1}{n} & \text{ if } \sigma(1)=1, \\
             -\frac{1}{n}, & \text{ otherwise.}
        \end{cases}
    \]
    Since these two values are nonzero, $r_1$ being sufficiently large ensures that $\langle \omega_1 \; | \; \sigma^{-1}(\rho) \rangle$ has the same sign as $\langle \omega_1 \;|\; \sigma^{-1}(\omega_1) \rangle$, {\it i.e.}, it is positive iff $\sigma(1)=1$.

    We have thus $\sigma(1)=1$ as a necessary condition for $\sigma(F) \subset H^+$.  Since there are $(n-1)!$ permutations in $\mathfrak{S}_n$ such that $\sigma(1)=1$, this condition is also sufficient.  (By~\eqref{eq:pin_mob}, the Möbius number of $\Pi_n$ is $(-1)^n (n-1)!$.)
\end{proof}

\begin{rema}
    In the situation of the previous proposition, $\CCP$ is convex in $\CC$.  Indeed, the condition $\sigma(1)=1$ is characterized by avoiding the inversions $(1i)$ for $2\leq i \leq n$.  It would be interesting to identify other situations where convexity holds, but none has been identified so far.
\end{rema}

Although we don't have any Coxeter theoretic result about the dimension of $\bxi$, the specific case of $\mathfrak{S}_n$ is worth investigating.   Let $D_n := \langle \xi_n | H_{1^n} \rangle$ be the dimension of $\xi_n$.  The first values (beginning at $n=1$) are 1, 1, 4, 33, 456, {\it etc}.  This sequence is known as~\cite[A002190]{oeis}, and the following result confirm that $(D_n)_{n\geq 1}$ match this OEIS entry.

\begin{prop} \label{prop:Sn}
    The dimension $D_n$ is the number of pairs $(\sigma,\tau) \in \mathfrak{S}_n^2$ such that $\tau(1)=1$ and $\des(\sigma) \subset \des(\tau)$.
    Moreover, we have:
    \begin{equation} \label{eq:gend_n}
        \sum_{n\geq 1} D_n \frac{x^n}{n!^2} = -\log(J_0(2\sqrt{x})),
    \end{equation}
    where $J_0$ is the {\it Bessel function} with parameter $0$, explicitly:    
    \[
        J_0(2\sqrt{x}) = \sum_{n\geq 0} \frac{(-x)^n}{n!^2}.
    \]
\end{prop}

\begin{proof}
    We use Theorem~\ref{theo4}, with $\rho$ as in Proposition~\ref{prop:CCP_Sn}.  This gives:
    \[
        D_n
            =
        \sum_{\tau \in \mathfrak{S}_n,\; \tau(F) \subset H^+ } \dim( \bphi_{\asc(\tau)} )
            =
        \sum_{\tau \in \mathfrak{S}_n,\; \tau(1)=1} \dim( \bphi_{\asc(\tau)} ).
    \]
    So, $D_n$ is the number of pairs $(\sigma,\tau)$ where $\tau(1)=1$, and $\sigma$ is a minimal length representative in $W/ W_{\asc(\tau)}$.  The latter condition is equivalent to $\asc(\sigma) \supset \asc(\tau)$, or equivalently $\des(\sigma) \subset \des(\tau)$.  This proves the first part of the proposition.

    Let us now get the generating function~\eqref{eq:gend_n}.  Recall that we have $\xi_n = (-1)^{n-1} \sum_{\pi\in\Pi_n} \mu(\pi) E_\pi$.  Note that the dimension of $E_\pi$ is the multinomial coefficient:
    \[
        \langle E_\pi | H_{1^n} \rangle
            =
        \binom{n}{\pi} := \frac{n!}{\prod_{\beta\in\pi} (\#\beta) !} ,
    \]
    (In general, $\dim\bphi_X = \frac{|W|}{|W_X|}$.)   It follows:
    \[
        D_n = \langle \xi_n \,| H_{1^n} \rangle
        =
        (-1)^{n-1} \sum_{\pi\in\Pi_n} \mu(\pi) \langle E_\pi \,|\,  H_{1^n} \rangle
        =
        (-1)^{n-1}  \sum_{\pi\in\Pi_n} \mu(\pi) \binom{n}{\pi}.
    \]
    We get:
    \[
        \sum_{n\geq 1} D_n \frac{x^n}{n!^2}
            =
        \sum_{n\geq 0} (-1)^{n-1} \sum_{\pi\in\Pi_n} \mu(\pi) \binom{n}{\pi}\frac{x^n}{n!^2}.
            =
        - \sum_{n\geq 0} \sum_{\pi\in\Pi_n} \frac{\mu(\pi)}{\prod_{\beta\in\pi} (\#\beta)! } \frac{(-x)^n}{n!}.
    \]
    By Lemma~\ref{lemm:momentcumulant}, we obtain~\eqref{eq:gend_n}.
\end{proof}

Aval, Boussicault, Bouvel and Silimbani study {\it complete non-ambiguous trees} in~\cite[Section~4]{avalboussicaultbouvelsilimbani}, and obtain the same generating function as in the previous proposition.  As they note, complete non-ambiguous trees provide the first combinatorial interpretation of the coefficients of the series $-\log(J_0(2\sqrt{x}))$.  Now, an open problem is to find an explicit bijection between:
\begin{itemize}
    \item complete non-ambiguous trees with $n-1$ internal vertices (see~\cite{avalboussicaultbouvelsilimbani} for the definition), 
    \item pairs $(\sigma,\tau) \in \mathfrak{S}_n^2$ such that $\tau(1)=1$ and $\des(\sigma) \subset \des(\tau)$.
\end{itemize}

\begin{rema}
    An integer sequence closely related to $(D_n)_{n\geq 1}$ is~\cite[A000275]{oeis}.  Let $D'_n$ be the number of pairs $(\sigma,\tau) \in \mathfrak{S}_n^2$ such that $\des(\sigma) \subset \des(\tau)$. Again, the doubly exponential generating function of these numbers involves the Bessel function:
    \begin{equation} \label{eq:J0inv}
        \sum_{n \geq 0} D'_n \frac{z^n}{n!^2} = \frac{1}{J_0(2\sqrt{x})}.
    \end{equation}
    See~\cite{carlitz}.  By comparing~\eqref{eq:gend_n} and~\eqref{eq:J0inv}, we see that
    \[
        \sum_{n \geq 0} D_n \frac{z^n}{n!^2}
            =
        \log\bigg(\sum_{n \geq 0} D'_n \frac{z^n}{n!^2}\bigg)
    \]
    It is an interesting exercise to find a combinatorial proof of this identity, using pairs of permutations $(\sigma,\tau)$ as above.  The idea is to use Foata's fundamental transformation (to get a statement in terms of excedances rather than descents), then interpret the logarithm as usual in the context of exponential generating function (linking permutations to cyclic permutations, in particular).
\end{rema}

\begin{rema}\label{rem:equiv_h2}
As we described in Remark~\ref{rem:equiv_h}, the dimension $D_n$ of $\bxi$ agrees with the top entry $h_n$ of the $h$-vector of the order complex $\Delta(\bar\PP)$. The whole $h$-vector is a statistic on the number $(n!)^2\cdot (n-1)!/2^{n-1}$ in the case of the symmetric group $\mathfrak{S}_n$, which for $n=4$ is $432$:
\[
h(\Delta(\bar\PP_{\mathfrak{S}_4}))=(1,127,271,33).
\]
Following Remark~\ref{rem:equiv_h}, the entries of the equivariant flag $h$-vector for $\bar\PP_{\mathfrak{S}_4}$ are easy to calculate (we give them in terms of Schur functions):
\begin{alignat*}{5}
    h_{123}&= 6S_{1^4} + &6S_{2,1^2} + &3S_{2,2} + &S_{3,1} \quad &\ \\
    h_{12}&= 5S_{1^4} + &9S_{2,1^2} + &4S_{2,2} + &3S_{3,1} \quad &\ \\ 
    h_{13}&= 6S_{1^4} + &18S_{2,1^2} + &9S_{2,2} + &11S_{3,1} \quad &\ \\ 
    h_{23}&=  &12S_{2,1^2} + &9S_{2,2} + &17S_{3,1} + &6S_4\\ 
    h_1&= S_{1^4} + &3S_{2,1^2} + &2S_{2,2} + &3S_{3,1} \quad &\ \\ 
    h_2&=  & 6S_{2,1^2} + &6S_{2,2} + &12S_{3,1} + &5S_4\\ 
    h_3&= &\ & 3S_{2,2} + &7S_{3,1} + &6S_4\\ 
    h_{\emptyset}&=  &\ &\ &\ &\ S_4.
\end{alignat*}
It is easy to see that the dimension of $h_{123}$ is $6+6\cdot 3+3\cdot 2 +3  =33$ as expected. Moreover (recall that $\omega$ is the involution on symmetric functions that correspond to tensoring with sign character):
\[
    \omega(h_{123}) 
        = 
    H_{2,1}+ 2 H_{2,2} + 2H_{3,1} + H_4,
\]
which agrees with the statement of Theorem~\ref{theo4} and Proposition~\ref{prop:CCP_Sn}. However, 
\[
    \omega(h_{13})
    = 
    11H_{2,1^2} - 2H_{2,2} - 2 H_{3,1} - H_4,
\]
is \emph{not} $H$-positive.
\end{rema}

\section{Final remarks}

Let us end the article with a few open questions.  

First, note that there is an extra symmetry on the poset of parabolic subgroups.  Since parabolic subgroups are closed under conjugation, the group $W^2$ acts on $\PP$ by
\[
    (w_1,w_2) \cdot wW_X := w_1wW_X w_2^{-1}
        =
    w_1w w_2^{-1} (w_2W_X w_2^{-1})
    .
\]
As a $W^2$-set, $\PP$ can be seen to be equivalent to
\[
    \biguplus_{X\in\Theta} W^2 / \delta(N(W_X))
\]
where $\delta:W\to W^2$ is the diagonal embedding.  Note that this isn't an element of the parabolic Burnside ring of $W^2$, as $\delta(N(W_X)$ might not be a parabolic subgroup.  It would be interesting to say more about the character of $W^2$ acting on $\PP$ and its homology.

It would be interesting to understand when $\CCP$ is convex in $\CC$.  The only known example so far is the one described in Proposition~\ref{prop:CCP_Sn}.  For example, if there is no other instances of convexity, there could be a somewhat conceptual explanation other than checking all cases of the finite type classification. 

Eventually, it could be interesting to consider the parabolic coset poset of complex reflection groups, or other Coxeter groups beyond finite type.




\begin{figure}
    \begin{center}
    \begin{tikzpicture}[scale=1.8]
        \draw[fill,color=lightgray] (1,0) arc(0:54.7:1) arc(114.7:150:2) arc(210:245.3:2) arc(305.3:360:1);
        \draw[fill,color=gray] (0.318,0) arc(30:45:2) arc(135:150:2);
        \draw (0,0) circle (1);
        \draw (4,0) -- (-4,0);
        \draw (1.414,-1) circle (2);
        \draw (1.414,1) circle (2);
        \draw (-1.414,1) circle (2);
        \draw (-1.414,-1) circle (2);
        \draw[line width=1mm,color=red] (0.364,0.081) circle (1.067);
    \end{tikzpicture}
    \null\vspace{2cm}
    \begin{tikzpicture}[scale=3]
        \draw[fill,color=lightgray] (1,0) arc(0:54.7:1) arc(114.7:150:2) arc(210:245.3:2) arc(305.3:360:1);
        \draw[fill,color=gray] (0.318,0) arc(30:45:2) arc(135:150:2);
        \draw (1,0) arc(0:54.7:1) arc(114.7:150:2) arc(210:245.3:2) arc(305.3:360:1) -- (-0.318,0);
        \draw (0.5773,0.8165) arc(-5.3:-45:2);
        \draw (0.5773,-0.8165) arc(5.3:45:2);
        \draw[fill,color=red] (0.318,0) circle(0.02);
        \draw[fill,color=blue] (-0.318,0) circle(0.02);
        \draw[fill,color=green] (0,0.414) circle(0.02);
        \draw[fill,color=green] (0,-0.414) circle(0.02);
        \draw[fill,color=green] (1,0) circle(0.02);
        \draw[fill,color=blue] (0.577,0.816) circle(0.02);
        \draw[fill,color=blue] (0.577,-0.816) circle(0.02);
    \end{tikzpicture}
    \hspace{3cm}
    \begin{tikzpicture}[scale=2]
        \useasboundingbox (0,-1) rectangle (2,1);
        \draw (0,0) -- (2,0);
        \node[fill,circle,color=blue] at(0,0){};
        \node[fill,circle,color=green] at(1,0){};
        \node[fill,circle,color=red] at(2,0){};
    \end{tikzpicture}
    \end{center}
    \null\vspace{0.4cm}
    \caption{The complex $\CCP$ in type $A_3$, first example.\label{fig:ccpA3_1}}
\end{figure}

\begin{figure}
    \begin{center}
     \begin{tikzpicture}[scale=1.8]
        \draw[fill,color=lightgray] (1,0) arc(0:125.3:1) arc(185.3:225:2) arc(315:330:2);
        \draw[fill,color=gray] (0.318,0) arc(30:45:2) arc(135:150:2);
        \draw (0,0) circle (1);
        \draw (4,0) -- (-4,0);
        \draw (1.414,-1) circle (2);
        \draw (1.414,1) circle (2);
        \draw (-1.414,1) circle (2);
        \draw (-1.414,-1) circle (2);
        \draw[line width=1mm,color=red] (0.1923,0.3077) circle (1.0638);
    \end{tikzpicture}
    \null\vspace{2cm}
    \begin{tikzpicture}[scale=3]
        \draw[fill,color=lightgray] (1,0) arc(0:125.3:1) arc(185.3:225:2) arc(315:330:2);
        \draw[fill,color=gray] (0.318,0) arc(30:45:2) arc(135:150:2);
        \draw (1,0) arc(0:125.3:1) arc(185.3:225:2) arc(315:354.7:2);
        \draw (0.318,0) arc(30:65.3:2);
        \draw (-0.318,0) arc(150:114.7:2);
        \draw (-0.318,0) -- (1,0);
        \draw[fill,color=red] (0.318,0) circle(0.02);
        \draw[fill,color=blue] (-0.318,0) circle(0.02);
        \draw[fill,color=green] (0,0.414) circle(0.02);
        \draw[fill,color=green] (0,-0.414) circle(0.02);
        \draw[fill,color=green] (1,0) circle(0.02);
        \draw[fill,color=blue] (0.577,0.816) circle(0.02);
        \draw[fill,color=red] (-0.577,0.816) circle(0.02);
    \end{tikzpicture}
    \hspace{3cm}
    \begin{tikzpicture}[scale=2]
        \useasboundingbox (0,-1) rectangle (2,1);
        \draw (0,0) -- (2,0);
        \node[fill,circle,color=blue] at(0,0){};
        \node[fill,circle,color=green] at(1,0){};
        \node[fill,circle,color=red] at(2,0){};
    \end{tikzpicture}
    \end{center}
    \null\vspace{0.4cm}
    \caption{The complex $\CCP$ in type $A_3$, second example.\label{fig:ccpA3_2}}
\end{figure}



\begin{figure}
    \begin{center}
    \begin{tikzpicture}[scale=1.8]
        \draw[fill,color=gray] (0.318,0) arc(30:45:2) arc(135:150:2);
        \draw[color=green] (0,-1) circle (1.414);
        \draw[color=green] (0,1) circle (1.414);
        \draw[color=green] (0,-4) -- (0,4);
        \draw[color=green] (-4.243,-3) -- (4.243,3);
        \draw[color=green] (-4.243,3) -- (4.243,-3);
        \draw[color=green] (0.707,0) circle (1.225);
        \draw[color=green] (-0.707,0) circle (1.225);
        \draw[color=green] (0,-1) circle (1.414);
        \draw[color=green] (0,1) circle (1.414);
        \draw[color=green] (0,-4) -- (0,4);
        \draw[color=green] (-4.243,-3) -- (4.243,3);
        \draw[color=green] (-4.243,3) -- (4.243,-3);
        \draw (0,0) circle (1);
        \draw (4,0) -- (-4,0);
        \draw (1.414,-1) circle (2);
        \draw (1.414,1) circle (2);
        \draw (-1.414,1) circle (2);
        \draw (-1.414,-1) circle (2);
    \end{tikzpicture}
    \end{center}
    \null\vspace{0.4cm}
    \caption{The dual arrangement in type $A_3$.\label{fig:dualA3}}
\end{figure}

%
\begin{figure}
    \begin{center}
    \begin{tikzpicture}[scale=2.4]
        \draw[fill,color=lightgray] (0,1) arc(90:-45:1) -- (0.366,-0.366) arc(-75:-105:1.414) arc(-165:-225:1.414);
        \draw[fill,color=gray] (0,0) --(0.366,0.366) arc(15:0:1.414) -- (0,0);
        \draw (0,0) circle (1);
        \draw (-1,0) circle (1.414);
        \draw (1,0) circle (1.414);
        \draw (0,-1) circle (1.414);
        \draw (0,1) circle (1.414);
        \draw (0,-2.8) -- (0,2.8);
        \draw (-2.8,0) -- (2.8,0);
        \draw (-1.9,1.9) -- (1.9,-1.9);
        \draw (-1.9,-1.9) -- (1.9,1.9);
        \draw[line width=1mm,color=red] (0.525,0.245) circle (1.156);
    \end{tikzpicture}
    \null\vspace{1cm}
    \begin{tikzpicture}[scale=2.4]
        \draw[fill,color=lightgray] (0,1) arc(90:-45:1) -- (0.366,-0.366) arc(-75:-105:1.414) arc(-165:-225:1.414);
        \draw[fill,color=gray] (0,0) --(0.366,0.366) arc(15:0:1.414) -- (0,0);
        \draw (0,1) arc (45:-15:1.414);
        \draw (0,1) arc (90:-45:1) -- (0.366,-0.366) arc (-75:-105:1.414) arc (-165:-225:1.414);
        \draw (-0.414,0) -- (1,0);
        \draw (0,-0.414) -- (0,1);
        \draw (-0.366,-0.366) -- (0.707,0.707);
        \draw (0.366,-0.366) --(-0.366,0.366) arc(-255:-315:1.414) arc(-45:-75:1.414);
        \draw[fill,color=red] (0,0) circle(0.02);
        \draw[fill,color=blue] (0.366,0.366) circle(0.02);
        \draw[fill,color=blue] (0.366,-0.366) circle(0.02);
        \draw[fill,color=blue] (-0.366,0.366) circle(0.02);
        \draw[fill,color=blue] (-0.366,-0.366) circle(0.02);
        \draw[fill,color=green] (-0.414,0) circle(0.02);
        \draw[fill,color=green] (0.414,0) circle(0.02);
        \draw[fill,color=green] (0,-0.414) circle(0.02);
        \draw[fill,color=green] (0,0.414) circle(0.02);
        \draw[fill,color=green] (0.707,0.707) circle(0.02);
        \draw[fill,color=green] (0.707,-0.707) circle(0.02);
        \draw[fill,color=red] (0,1) circle(0.02);
        \draw[fill,color=red] (1,0) circle(0.02);
    \end{tikzpicture}
    \hspace{3cm}
    \begin{tikzpicture}[scale=2]
        \useasboundingbox (0,-1) rectangle (2,1);
        \draw (0,0) -- (2,0);
        \node[fill,circle,color=blue] at(0,0){};
        \node[fill,circle,color=green] at(1,0){};
        \node[fill,circle,color=red] at(2,0){};
        \node at (0.5,0.15){4};
    \end{tikzpicture}
    \null\vspace{0.4cm}
    \end{center}
    \caption{The complex $\CCP$ in type $B_3$, first example.\label{fig:ccpB3_1}}
\end{figure}

\begin{figure}
    \begin{center}
    \begin{tikzpicture}[scale=2.4]
        \draw[fill,color=lightgray] (1,0) arc(-45:15:1.414) arc(75:180:1.414) -- (0,0) --(0,-0.414) arc(-90:-75:1.414) -- (0.707,-0.707) arc(-45:0:1);
        \draw[fill,color=gray] (0,0) --(0.366,0.366) arc(15:0:1.414) -- (0,0);
        \draw (0,0) circle (1);
        \draw (-1,0) circle (1.414);
        \draw (1,0) circle (1.414);
        \draw (0,-1) circle (1.414);
        \draw (0,1) circle (1.414);
        \draw (0,-2.8) -- (0,2.8);
        \draw (-2.8,0) -- (2.8,0);
        \draw (-1.9,1.9) -- (1.9,-1.9);
        \draw (-1.9,-1.9) -- (1.9,1.9);
        \draw[line width=1mm,color=red] (0.847,0.395) circle (1.369);
    \end{tikzpicture}
    \null\vspace{1cm}
    \begin{tikzpicture}[scale=2.4]
        \draw[fill,color=lightgray] (1,0) arc(-45:15:1.414) arc(75:180:1.414) -- (0,0) --(0,-0.414) arc(-90:-75:1.414) -- (0.707,-0.707) arc(-45:0:1);
        \draw[fill,color=gray] (0,0) --(0.366,0.366) arc(15:0:1.414) -- (0,0);
        \draw (0,1) arc (45:-15:1.414);
        \draw (0,1) arc (90:-45:1) -- (0.366,-0.366);
        \draw (-0.414,0) -- (1,0) arc(-45:15:1.414) arc(75:180:1.414);
        \draw (0,0) -- (1.366,1.366);
        \draw (0.366,-0.366) --(-0.366,0.366) arc(-255:-315:1.414) arc(-45:-90:1.414) --(0,1);
        \draw[fill,color=red] (0,0) circle(0.02);
        \draw[fill,color=blue] (0.366,0.366) circle(0.02);
        \draw[fill,color=blue] (0.366,-0.366) circle(0.02);
        \draw[fill,color=blue] (-0.366,0.366) circle(0.02);
        \draw[fill,color=blue] (1.366,1.366) circle(0.02);
        \draw[fill,color=green] (-0.414,0) circle(0.02);
        \draw[fill,color=green] (0.414,0) circle(0.02);
        \draw[fill,color=green] (0,-0.414) circle(0.02);
        \draw[fill,color=green] (0,0.414) circle(0.02);
        \draw[fill,color=green] (0.707,0.707) circle(0.02);
        \draw[fill,color=green] (0.707,-0.707) circle(0.02);
        \draw[fill,color=red] (0,1) circle(0.02);
        \draw[fill,color=red] (1,0) circle(0.02);
    \end{tikzpicture}
    \hspace{2.4cm}
    \begin{tikzpicture}[scale=2]
        \useasboundingbox (0,-1) rectangle (2,1);
        \draw (0,0) -- (2,0);
        \node[fill,circle,color=blue] at(0,0){};
        \node[fill,circle,color=green] at(1,0){};
        \node[fill,circle,color=red] at(2,0){};
        \node at (0.5,0.15){4};
    \end{tikzpicture}
    \null\vspace{0.4cm}
    \end{center}
    \caption{The complex $\CCP$ in type $B_3$, second example.\label{fig:ccpB3_2}}
\end{figure}


%
\begin{figure}
    \begin{center}
    \begin{tikzpicture}[scale=2.4]
        \draw[fill,color=gray] (0,0) --(0.366,0.366) arc(15:0:1.414) -- (0,0);
        \draw[color=green] (1,1) circle (1.73);
        \draw[color=green] (-1,1) circle (1.73);
        \draw[color=green] (1,-1) circle (1.73);
        \draw[color=green] (-1,-1) circle (1.73);
        \draw[color=green,line width=0.28mm] (0,0) circle (1);
        \draw[color=green,line width=0.28mm] (-1,0) circle (1.414);
        \draw[color=green,line width=0.28mm] (1,0) circle (1.414);
        \draw[color=green,line width=0.28mm] (0,-1) circle (1.414);
        \draw[color=green,line width=0.28mm] (0,1) circle (1.414);
        \draw[color=green,line width=0.28mm] (0,-2.8) -- (0,2.8);
        \draw[color=green,line width=0.28mm] (-2.8,0) -- (2.8,0);
        \draw[color=green,line width=0.28mm] (-2.5,2.5) -- (2.5,-2.5);
        \draw[color=green,line width=0.28mm] (-2.5,-2.5) -- (2.5,2.5);%
        \draw (0,0) circle (1);
        \draw (-1,0) circle (1.414);
        \draw (1,0) circle (1.414);
        \draw (0,-1) circle (1.414);
        \draw (0,1) circle (1.414);
        \draw (0,-2.8) -- (0,2.8);
        \draw (-2.8,0) -- (2.8,0);
        \draw (-2.5,2.5) -- (2.5,-2.5);
        \draw (-2.5,-2.5) -- (2.5,2.5);
    \end{tikzpicture}
    \null\vspace{0.4cm}
    \end{center}
    \caption{The dual arrangement in type $B_3$.\label{fig:dualB3}}
\end{figure}



\end{document}